\title{Strong Ramsey Games: Drawing on an infinite board}

\author{
Dan Hefetz
\thanks{Department of Computer Science, Hebrew University, Jerusalem 9190401 and School of Mathematical Sciences, Raymond and Beverly Sackler Faculty of Exact Sciences, Tel Aviv University, 6997801, Israel. Email: danny.hefetz@gmail.com.} 
\and Christopher Kusch 
\thanks{Institut f\"ur Mathematik und Informatik, Freie Universit\"at Berlin and Berlin Mathematical School, Germany. Email: c.kusch@zedat.fu-berlin.de. Research supported by a Berlin Mathematical School Phase II scholarship.}
\and Lothar Narins
\thanks{Institut f\"ur Mathematik und Informatik, Freie Universit\"at Berlin, Germany. Email: Narins@math.fu-berlin.de}
\and Alexey Pokrovskiy
\thanks{ETH Zurich, Switzerland. Email: DrAlexeyPokrovskiy@gmail.com}
\and Cl\'ement Requil\'e
\thanks{Institut f\"ur Mathematik und Informatik, Freie Universit\"at Berlin and Berlin Mathematical School, Germany. Research supported by the FP7-PEOPLE-2013-CIG project CountGraph (ref. 630749). Email: requile@math.fu-berlin.de} 
\and Amir Sarid
\thanks{School of Mathematical Sciences, Raymond and Beverly Sackler Faculty of Exact Sciences, Tel Aviv University, 6997801, Israel. Email: amirsar1@mail.tau.ac.il} 
}

\documentclass[11pt]{article}
\usepackage{graphicx}

\usepackage{amsmath,amsthm,amssymb,latexsym,color,epsfig,a4,enumerate}

\newtheorem{theorem}{Theorem} [section]
\newtheorem{lemma}[theorem]{Lemma}

\newtheorem{claim}[theorem]{Claim}

\newtheorem{observation}[theorem]{Observation}

\newtheorem{question}[theorem]{Question}

\begin{document}

\maketitle

\begin{abstract}
We consider the strong Ramsey-type game $\mathcal{R}^{(k)}(\mathcal{H}, \aleph_0)$, played on the edge set of the infinite complete $k$-uniform hypergraph $K^k_{\mathbb{N}}$. Two players, called FP (the first player) and SP (the second player), take turns claiming edges of $K^k_{\mathbb{N}}$ with the goal of building a copy of some finite predetermined $k$-uniform hypergraph $\mathcal{H}$. The first player to build a copy of $\mathcal{H}$ wins. If no player has a strategy to ensure his win in finitely many moves, then the game is declared a draw. 

In this paper, we construct a $5$-uniform hypergraph $\mathcal{H}$ such that $\mathcal{R}^{(5)}(\mathcal{H}, \aleph_0)$ is a draw. This is in stark contrast to the corresponding finite game $\mathcal{R}^{(5)}(\mathcal{H}, n)$, played on the edge set of $K^5_n$. Indeed, using a classical game-theoretic argument known as \emph{strategy stealing} and a Ramsey-type argument, one can show that for every $k$-uniform hypergraph $\mathcal{G}$, there exists an integer $n_0$ such that FP has a winning strategy for $\mathcal{R}^{(k)}(\mathcal{G}, n)$ for every $n \geq n_0$. 
\end{abstract}

\section{Introduction}

The theory of positional games on graphs and hypergraphs goes back to the seminal papers of Hales and Jewett~\cite{HJ} and of Erd\H{o}s and Selfridge~\cite{ES}. The theory has enjoyed explosive growth in recent years and has matured into an important area of combinatorics (see the monograph of Beck~\cite{TTT}, the recent monograph~\cite{HKSSbook} and the survey~\cite{Krivelevich}). There are several interesting types of positional games, the most natural of which are the so-called \emph{strong games}. 

Let $X$ be a (possibly infinite) set and let $\mathcal{F}$ be a family of finite subsets of $X$. The \emph{strong game} $(X, \mathcal{F})$ is played by two players, called FP (the first player) and SP (the second player), who take turns claiming previously unclaimed elements of the \emph{board} $X$, one element per move. The winner of the game is the \emph{first} player to claim all elements of a \emph{winning set} $A \in \mathcal{F}$. If no player wins the game after some finite number of moves, then the game is declared a \emph{draw}. A very simple but classical example of this setting is the game of Tic-Tac-Toe.

Unfortunately, strong games are notoriously hard to analyze and to date not much is known about them. A simple yet elegant game-theoretic argument, known as \emph{strategy stealing}, shows that FP is guaranteed at least a draw in any strong game. Moreover, using Ramsey Theory, one can sometimes prove that draw is impossible in a given strong game and thus FP has a winning strategy for this game. Note that these arguments are purely existential and thus even if we know that FP has a winning/drawing strategy for some game, we might not know what it is. Explicit winning strategies for FP in various natural strong games were devised in~\cite{FH} and in~\cite{FHkcon}. These strategies are based on fast winning strategies for \emph{weak} variants of the games in question. More on fast winning strategies can be found in~\cite{HKSS} and~\cite{CFGHL}.  

In this paper we study a natural family of strong games. For integers $n \geq q \geq 3$, consider the strong Ramsey game $\mathcal{R}(K_q, n)$. The board of this game is the edge set of $K_n$ and the winning sets are the copies of $K_q$ in $K_n$. As noted above, by strategy stealing, FP has a drawing strategy in $\mathcal{R}(K_q, n)$ for every $n$ and $q$. Moreover, it follows from Ramsey's famous Theorem~\cite{Ramsey} (see also~\cite{GRS} and~\cite{CFSsurvey} for numerous related results) that, for every $q$, there exists an $n_0$ such that $\mathcal{R}(K_q, n)$ has no drawing position and is thus FP's win for every $n \geq n_0$. An explicit winning strategy for FP in $\mathcal{R}(K_q, n)$ is currently known (and is very easy to find) only for $q = 3$ (and every $n \geq 5$). Moreover, for every $q \geq 4$, we do not know what is the smallest $n_0 = n_0(q)$ such that $\mathcal{R}(K_q, n)$ is FP's win for every $n \geq n_0$. Determining this value seems to be extremely hard even for relatively small values of $q$.             

Consider now the strong game $\mathcal{R}(K_q, \aleph_0)$. Its board is the edge set of the countably infinite complete graph $K_{\mathbb{N}}$ and its winning sets are the copies of $K_q$ in $K_{\mathbb{N}}$. Even though the board of this game is infinite, strategy stealing still applies, i.e., FP has a strategy which ensures that SP will never win $\mathcal{R}(K_q, \aleph_0)$. Clearly, Ramsey's Theorem applies as well, i.e., any red/blue colouring of the edges of $K_{\mathbb{N}}$ yields a monochromatic copy of $K_q$. Hence, as in the finite version of the game, one could expect to combine these two arguments to deduce that FP has a winning strategy in $\mathcal{R}(K_q, \aleph_0)$. The only potential problem with this reasoning is that, by making infinitely many threats (which are idle, as he cannot win), SP might be able to delay FP indefinitely, in which case the game would be declared a draw. As with the finite version, $\mathcal{R}(K_3, \aleph_0)$ is an easy win for FP. The question whether $\mathcal{R}(K_q, \aleph_0)$ is a draw or FP's win is wide open for every $q \geq 4$. In fact, using different terminology, it was posed by Beck~\cite{TTT} as one of his ``7 most humiliating open problems'', where he considers even the case $q = 5$ to be ``hopeless'' (see also~\cite{Leader} and~\cite{Bowler} for related problems).

Playing Ramsey games, we do not have to restrict our attention to cliques, or even to graphs for that matter. For every integer $k \geq 2$ and every $k$-uniform hypergraph $\mathcal{H}$, we can study the finite strong Ramsey game $\mathcal{R}^{(k)}(\mathcal{H}, n)$ and the infinite strong Ramsey game $\mathcal{R}^{(k)}(\mathcal{H}, \aleph_0)$. The board of the finite game $\mathcal{R}^{(k)}(\mathcal{H}, n)$ is the edge set of the complete $k$-uniform hypergraph $K_n^k$ and the winning sets are the copies of $\mathcal{H}$ in $K_n^k$. As in the graph case, strategy stealing and Hypergraph Ramsey Theory (see, e.g.,~\cite{CFS}) shows that FP has winning strategies in $\mathcal{R}^{(k)}(\mathcal{H}, n)$ for every $\mathcal{H}$ and every sufficiently large $n$. The board of the infinite game $\mathcal{R}^{(k)}(\mathcal{H}, \aleph_0)$ is the edge set of the countably infinite complete $k$-uniform hypergraph $K_{\mathbb{N}}^k$ and the winning sets are the copies of $\mathcal{H}$ in $K_{\mathbb{N}}^k$. As in the graph case, strategy stealing shows that FP has drawing strategies in $\mathcal{R}^{(k)}(\mathcal{H}, \aleph_0)$ for every $\mathcal{H}$. Hence, here too one could expect to combine strategy stealing and Hypergraph Ramsey Theory to deduce that FP has a winning strategy in $\mathcal{R}^{(k)}(\mathcal{H}, \aleph_0)$ for every $\mathcal{H}$.

Our main result shows that, while it might be true that $\mathcal{R}(K_q, \aleph_0)$ is FP's win for any $q \geq 4$, basing this solely on strategy stealing and Ramsey Theory is ill-founded.   

\begin{theorem} \label{th::main}
There exists a $5$-uniform hypergraph ${\mathcal H}$ such that the strong game $\mathcal{R}^{(5)}(\mathcal{H}, \aleph_0)$ is a draw.
\end{theorem} 

Apart from being very surprising, Theorem~\ref{th::main} might indicate that strong Ramsey games are even more complicated than we originally suspected. We discuss this further in Section~\ref{sec::openprob}.                    

The rest of this paper is organized as follows. In Section~\ref{sec::notation} we introduce some basic notation and terminology that will be used throughout this paper. In Section~\ref{sec::sufficientCondition} we prove that $\mathcal{R}^{(k)}(\mathcal{H}, \aleph_0)$ is a draw whenever $\mathcal{H}$ is a $k$-uniform hypergraph which satisfies certain conditions. Using the results of Section~\ref{sec::sufficientCondition}, we construct in Section~\ref{sec::example} a $5$-uniform hypergraph ${\mathcal H}_5$ for which $\mathcal{R}^{(5)}(\mathcal{H}_5, \aleph_0)$ is a draw, thus proving Theorem~\ref{th::main}. Finally, in Section~\ref{sec::openprob} we present some open problems.    

\section{Notation and terminology} \label{sec::notation}
Let $\mathcal{H}$ be a $k$-uniform hypergraph. We denote its vertex set by $V(\mathcal{H})$ and its edge set by $E(\mathcal{H})$. The \emph{degree} of a vertex $x \in V(\mathcal{H})$ in $\mathcal{H}$, denoted by $d_{\mathcal H}(x)$, is the number of edges of $\mathcal{H}$ which are incident with $x$. The \emph{minimum degree} of $\mathcal{H}$, denoted by $\delta(\mathcal{H})$, is $\min \{d_{\mathcal H}(u) : u \in V(\mathcal{H})\}$. We will often use the terminology \emph{$k$-graph} or simply \emph{graph} rather than $k$-uniform hypergraph.

A \emph{tight path} is a $k$-graph with vertex set $\{u_1, \ldots, u_t\}$ and edge set $e_1, \ldots, e_{t-k+1}$ such that $e_i = \{u_i, \ldots, u_{i+k-1}\}$ for every $1 \leq i \leq t-k+1$. The \emph{length} of a tight path is the number of its edges.
           
We say that a $k$-graph $\mathcal{F}$ has a \emph{fast winning strategy} if a player can build a copy of $\mathcal{F}$ in $|E(\mathcal{F})|$ moves (note that this player is not concerned about his opponent building a copy of $\mathcal{F}$ first).      

\section{Sufficient conditions for a draw} \label{sec::sufficientCondition}
In this section we list several conditions on a $k$-graph $\mathcal{H}$ which suffice to ensure that $\mathcal{R}^{(k)}(\mathcal{H}, \aleph_0)$ is a draw. 

\begin{theorem} \label{th::HypergraphProperties}
Let $\mathcal{H}$ be a $k$-graph which satisfies all of the following properties: 
\begin{description}
\item [(i)] $\mathcal{H}$ has a degree $2$ vertex $z$;
\item [(ii)] $\delta(\mathcal{H} \setminus \{z\}) \geq 3$ and $d_{\mathcal H}(u) \geq 4$ for every $u \in V({\mathcal H}) \setminus \{z\}$;  
\item [(iii)] $\mathcal{H} \setminus \{z\}$ has a fast winning strategy;
\item [(iv)] For every two edges $e, e' \in \mathcal{H}$, if $\phi : V(\mathcal{H} \setminus \{e, e'\}) \longrightarrow V(\mathcal{H})$ is a monomorphism, then $\phi$ is the identity;
\item [(v)] $e \cap r \neq \emptyset$ and $e \cap g \neq \emptyset$ holds for every edge $e \in \mathcal{H}$, where $r$ and $g$ are the two edges incident with $z$ in $\mathcal{H}$. 
\item [(vi)] $|V(\mathcal{H}) \setminus (r \cup g)| < k-1$. 
\end{description}
Then $\mathcal{R}^{(k)}(\mathcal{H}, \aleph_0)$ is a draw.
\end{theorem}

Before proving this theorem, we will introduce some more notation and terminology which will be used throughout this section. Let $e \in \mathcal{H}$ be an arbitrary edge, let $\mathcal{F}$ be a copy of $\mathcal{H} \setminus \{e\}$ in $K_{\mathbb{N}}^k$ and let $e' \in K_{\mathbb{N}}^k$ be an edge such that $\mathcal{F} \cup \{e'\} \cong \mathcal{H}$. If $e'$ is free, then it is said to be a \emph{threat} and $\mathcal{F}$ is said to be \emph{open}. If $\mathcal{F}$ is not open, then it is said to be \emph{closed}. Moreover, $e'$ is called a \emph{standard threat} if it is a threat and $e \in \{r,g\}$. Similarly, $e'$ is called a \emph{special threat} if it is a threat and $e \notin \{r,g\}$.     

Next, we state and prove two simple technical lemmata.

\begin{lemma} \label{lem::missing1edge}
Let $\mathcal{H}$ be a $k$-graph which satisfies Properties (i), (ii) and (iv) from Theorem~\ref{th::HypergraphProperties}. Then, for every edge $e \in \mathcal{H}$, if $\phi : V(\mathcal{H} \setminus \{e\}) \longrightarrow V(\mathcal{H})$ is a monomorphism, then $\phi$ is the identity. 
\end{lemma}

\begin{proof}  
Fix an arbitrary edge $e \in \mathcal{H}$ and an arbitrary monomorphism $\phi : V(\mathcal{H} \setminus \{e\}) \longrightarrow V(\mathcal{H})$. It follows by Properties (i) and (ii) that there exists an edge $f \in \mathcal{H} \setminus \{e\}$ such that $V(\mathcal{H} \setminus \{e, f\}) = V(\mathcal{H})$. Hence, $\phi$ equals its restriction to $V(\mathcal{H} \setminus \{e, f\})$ which is the identity by Property (iv).       
\end{proof} 

\begin{lemma} \label{lem::uniquerg} 
Let $\mathcal{H}$ be a $k$-graph which satisfies Properties (i) and (iv) from Theorem~\ref{th::HypergraphProperties}. For any given copy $\mathcal{H}'$ of $\mathcal{H} \setminus \{z\}$ in $K^k_{\mathbb{N}}$ and any vertex $x \in V(K^k_{\mathbb{N}}) \setminus V(\mathcal{H}')$, there exists a unique pair of edges $r', g' \in K^k_{\mathbb{N}}$ such that $x \in r' \cap g'$ and $\mathcal{H}' \cup \{r', g'\} \cong \mathcal{H}$.
\end{lemma}

\begin{proof}
Let $\mathcal{H}'$ be an arbitrary copy of $\mathcal{H} \setminus \{z\}$ in $K^k_{\mathbb{N}}$ and let $x \in V(K^k_{\mathbb{N}}) \setminus V(\mathcal{H}')$ be an arbitrary vertex. It is immediate from the definition of $\mathcal{H}'$ and Property (i) that there are edges $r', g' \in E(K^k_{\mathbb{N}})$ such that $x \in r' \cap g'$ and $\mathcal{H}' \cup \{r', g'\} \cong \mathcal{H}$. Suppose for a contradiction that there are edges $r'', g'' \in E(K^k_{\mathbb{N}})$ such that $\{r'', g''\} \neq \{r', g'\}$, $x \in r'' \cap g''$ and $\mathcal{H}' \cup \{r'', g''\} \cong \mathcal{H}$. Let $\phi : V(\mathcal{H}' \cup \{r',g'\}) \rightarrow V(\mathcal{H}' \cup \{r'', g''\})$ be an arbitrary isomorphism. The restriction of $\phi$ to $V(\mathcal{H}')$ is clearly a monomorphism and is thus the identity by Property (iv). Since $x$ is the only vertex in $(r' \cap g') \setminus V(\mathcal{H}')$ and in $(r'' \cap g'') \setminus V(\mathcal{H}')$, it follows that $\phi$ itself is the identity and thus $\{r',g'\} = \{r'',g''\}$ contrary to our assumption.     
\end{proof}

We are now in a position to prove the main result of this section.

\noindent \emph{Proof of Theorem~\ref{th::HypergraphProperties}}.
Let $\mathcal{H}$ be a $k$-graph which satisfies the conditions of the theorem and let $m = |E(\mathcal{H})|$. At any point during the game, let $\mathcal{G}_1$ denote FP's current graph and let $\mathcal{G}_2$ denote SP's current graph. We will describe a drawing strategy for SP. We begin by a brief description of its main ideas and then detail SP's moves in each case. The strategy is divided into three stages. In the first stage SP quickly builds a copy of $\mathcal{H} \setminus \{z\}$, in the second stage SP defends against FP's threats, and in the third stage (which we might never reach) SP makes his own threats.

\bigskip

\noindent \textbf{Stage I:} Let $e_1$ denote the edge claimed by FP in his first move. In his first $m-2$ moves, SP builds a copy of $\mathcal{H} \setminus \{z\}$ which is vertex-disjoint from $e_1$. SP then proceeds to Stage II.     

\medskip

\noindent \textbf{Stage II:} Immediately before each of SP's moves in this stage, he checks whether there are a subgraph $\mathcal{F}_1$ of $\mathcal{G}_1$ and a free edge $e' \in K^k_{\mathbb{N}}$ such that $\mathcal{F}_1 \cup \{e'\} \cong \mathcal{H}$. If such $\mathcal{F}_1$ and $e'$ exist, then SP claims $e'$ (we will show later that, if such $\mathcal{F}_1$ and $e'$ exist, then they are unique). Otherwise, SP proceeds to Stage III.

\medskip

\noindent \textbf{Stage III:} Let $\mathcal{F}_2$ be a copy of $\mathcal{H} \setminus \{z\}$ in $\mathcal{G}_2$ and let $z'$ be an arbitrary vertex of $K^k_{\mathbb{N}} \setminus (\mathcal{G}_1 \cup \mathcal{G}_2)$. Let $r', g' \in K^k_{\mathbb{N}}$ be free edges such that $z' \in r' \cap g'$ and $\mathcal{F}_2 \cup \{r', g'\} \cong \mathcal{H}$. If, once SP claims $r'$, FP cannot make a threat by claiming $g'$, then SP claims $r'$. Otherwise he claims $g'$.  	

\bigskip

It readily follows by Property (iii) that SP can play according to Stage I of the strategy (since $K^k_{\mathbb{N}}$ is infinite, it is evident that SP's graph can be made disjoint from $e_1$). It is obvious from its description that SP can play according to Stage II of the strategy. Finally, since SP builds a copy of $\mathcal{H} \setminus \{z\}$ in Stage I and since $K^k_{\mathbb{N}}$ is infinite, it follows that SP can play according to Stage III of the strategy as well.    

It thus remains to prove that the proposed strategy ensures at least a draw for SP. Since, trivially, FP cannot win the game in less than $m$ moves, this will readily follow from the next three lemmata which correspond to three different options for FP's $(m-1)$th move.

\begin{lemma} \label{lem::noThreat}
If FP's $(m-1)$th move is not a threat, then he cannot win the game. 
\end{lemma} 

\begin{proof}
Assume that SP does not win the game. We will prove that, under this assumption, not only does FP not win the game, but in fact he does not even make a single threat throughout the game. We will prove by induction on $i$ that the following two properties hold immediately after FP's $i$th move for every $i \geq m-1$. 
\begin{description}
\item [(a)] FP has no threat.
\item [(b)] Let ${\mathcal G}'_1$ denote FP's graph immediately after his $(m-1)$th move. Then ${\mathcal G}_1 \setminus {\mathcal G}'_1$ consists of $i-m+1$ edges $e_m, \ldots, e_i$, where, for every $m \leq j \leq i$, $e_j$ contains a vertex $z_j$ such that $d_{{\mathcal G}_1}(z_j) = 1$.  
\end{description}

Properties (a) and (b) hold for $i = m-1$ by assumption. Assume they hold for some $i \geq m-1$; we will prove they hold for $i+1$ as well. Since FP's $(m-1)$th move is not a threat, SP's $i$th move is played in Stage III. By the description of Stage III, in his $i$th move SP claims an edge $e' \in \{r', g'\}$, where both $r'$ and $g'$ contain a vertex $z'$ which is isolated in ${\mathcal G}_1$. If FP does not respond by claiming the unique edge of $\{r', g'\} \setminus \{e'\}$ in his $(i+1)$th move, then SP will claim it in his $(i+1)$th move and win the game contrary to our assumption (by Property (a), FP had no threat before SP's $i$th move and thus cannot complete a copy of $\mathcal{H}$ in one move). It follows that Property (b) holds immediately after FP's $(i+1)$th move. Suppose for a contradiction that Property (a) does not hold, i.e., that FP makes a threat in his $(i+1)$th move. As noted above, in his $i$th move, SP claims either $r'$ or $g'$ and, by our assumption that Property (a) does not hold immediately after FP's $(i+1)$th move, in either case FP's response is a threat. Hence, immediately after FP's $(i+1)$th move, there exist free edges $r''$ and $g''$ and copies ${\mathcal F}^r$ and ${\mathcal F}^g$ of ${\mathcal H} \setminus \{z\}$ in ${\mathcal G}_1$ such that ${\mathcal F}^r \cup \{r', g''\} \cong {\mathcal H}$ and ${\mathcal F}^g \cup \{r'', g'\} \cong {\mathcal H}$. By Property (ii) and since, by the induction hypothesis, Property (b) holds for $i$, we have ${\mathcal F}^r \subseteq {\mathcal G}'_1$ and ${\mathcal F}^g \subseteq {\mathcal G}'_1$. Suppose for a contradiction that $e_1 \in {\mathcal F}^r$. Since ${\mathcal F}^r \cup \{r'\}$ is a threat, with $z' \in r'$ in the role of $z$, it follows by Property (v) that $r' \cap e_1 \neq \emptyset$. However, SP could have created a threat by claiming $r'$ in his $i$th move which, by Stages I and III of SP's strategy, implies that $r' \cap e_1 = \emptyset$. Hence $e_1 \notin {\mathcal F}^r$ and an analogous argument shows that $e_1 \notin {\mathcal F}^g$. Since $|E({\mathcal G}'_1) \setminus \{e_1\}| = m-2$, it follows that ${\mathcal F}^r = \mathcal{G}'_1 \setminus \{e_1\} = {\mathcal F}^g$. Therefore, by Lemma~\ref{lem::uniquerg} we have $\{r', g''\} = \{r'', g'\}$. Since, clearly $r' \neq g'$, it follows that $\{r'', g''\} = \{r', g'\}$ contrary to our assumption that both $r''$ and $g''$ were free immediately before FP's $(i+1)$th move. We conclude that Property (a) holds immediately after FP's $(i+1)$th move as well.                \end{proof} 

\begin{lemma} \label{lem::specialThreat}
If FP's $(m-1)$th move is a special threat, then he cannot win the game. 
\end{lemma}

\begin{proof}
Assume that SP does not win the game. We will prove that, under this assumption, FP does not win the game. We begin by showing that he does not win the game in his $m$th move. Let $e'$ be a free edge such that $\mathcal{G}_1 \cup \{e'\} \cong \mathcal{H}$. Playing according to the proposed strategy, SP responds to this threat by claiming $e'$. Let $f'$ denote the edge FP claims in his $m$th move. Suppose for a contradiction that, by claiming $f'$, FP completes a copy of $\mathcal{H}$. Note that $(\mathcal{G}_1 \setminus \{f'\}) \cup \{e'\} \cong \mathcal{H}$ and so there exists an isomorphism $\phi : V((\mathcal{G}_1 \setminus \{f'\}) \cup \{e'\}) \rightarrow V(\mathcal{G}_1)$. The restriction of $\phi$ to $V(\mathcal{G}_1 \setminus \{f'\})$ is clearly a monomorphism and is thus the identity by Lemma~\ref{lem::missing1edge}. However, $V((\mathcal{G}_1 \setminus \{f'\}) \cup \{e'\}) = V(\mathcal{G}_1 \setminus \{f'\})$ and so $\phi$ itself is the identity. It follows that $e' \in \mathcal{G}_1$ and thus $e' \in \mathcal{G}_1 \cap \mathcal{G}_2$ which is clearly a contradiction. We conclude that indeed FP does not win the game in his $m$th move. Next, we prove that, in his $m$th move, FP does not even make a threat. Suppose for a contradiction that by claiming $f'$ in his $m$th move, FP does create a threat. Immediately after FP's $m$th move, let $f'' \in {\mathcal G}_1$ and $f''' \in K_{\mathbb{N}}^k \setminus ({\mathcal G}_1 \cup {\mathcal G}_2)$ be edges such that ${\mathcal H}' := ({\mathcal G}_1 \setminus \{f''\}) \cup \{f'''\} \cong {\mathcal H}$. Recall that ${\mathcal H}'' := ({\mathcal G}_1 \setminus \{f'\}) \cup \{e'\} \cong {\mathcal H}$ as well. Let $\phi : V({\mathcal H}'') \rightarrow V(\mathcal{H}')$ be an isomorphism. The restriction of $\phi$ to $V({\mathcal H}'' \setminus \{e', f''\})$ is clearly a monomorphism and is thus the identity by Property (iv). Since FP's $(m-1)$th move was a special threat, it follows that $V({\mathcal H}'' \setminus \{e', f''\}) = V({\mathcal H}'')$ and thus $\phi$ itself is the identity. Therefore $e' \in {\mathcal H}'$. Since $e' \neq f'''$ we then have $e' \in \mathcal{G}_1$ and thus $e' \in \mathcal{G}_1 \cap \mathcal{G}_2$ which is clearly a contradiction. We conclude that indeed FP does not make a threat in his $m$th move.  

It remains to prove that FP cannot win the game in his $i$th move for any $i \geq m+1$. We will prove by induction on $i$ that the following two properties hold immediately after FP's $i$th move for every $i \geq m$. 
\begin{description}
\item [(a)] FP has no threat.
\item [(b)] ${\mathcal G}_1$ contains at most one copy of ${\mathcal H} \setminus \{z\}$.  
\end{description}

Starting with the induction basis $i = m$, note that Property (a) holds by the paragraph above. Moreover, since FP's $(m-1)$th move is a special threat, immediately after this move, there exists a vertex $u$ of degree two in ${\mathcal G}_1$. By Property (ii), this vertex and the two edges incident with it cannot be a part of any copy of ${\mathcal H} \setminus \{z\}$ in ${\mathcal G}_1$ immediately after FP's $m$th move. Property (b) now follows since FP's graph contains only $m-2$ additional edges. Assume Properties (a) and (b) hold immediately after FP's $i$th move for some $i \geq m$; we will prove they hold after his $(i+1)$th move as well. As in the proof of Lemma~\ref{lem::noThreat}, we can assume that in his $(i+1)$th move FP claims either $r'$ or $g'$. Since both edges contain a vertex which was isolated in ${\mathcal G}_1$ immediately before FP's $(i+1)$th move, neither edge can be a part of a copy of ${\mathcal H} \setminus \{z\}$ in ${\mathcal G}_1$. Hence, Property (b) still holds. As in the proof of Lemma~\ref{lem::noThreat}, if FP does make a threat in his $(i+1)$th move, then ${\mathcal G}_1$ must contain two copies ${\mathcal F}^r \neq {\mathcal F}^g$ of ${\mathcal H} \setminus \{z\}$ contrary to Property (b). We conclude that Property (a) holds as well.            
\end{proof}

\begin{lemma} \label{lem::standardThreat}
If FP's $(m-1)$th move is a standard threat, then he cannot win the game.  
\end{lemma}

\begin{proof}
The basic idea behind this proof is that either FP continues making simple threats forever or, at some point, he makes a move which is not a standard threat. We will prove that, assuming SP does not win the game, in the former case there is always a unique threat which SP can block, and in the latter case, by making his own standard threats, SP can force FP to respond to these threats forever, without ever creating another threat of his own.  

We first claim that, if FP does win the game in some move $s$, then there must exist some $m \leq i < s$ such that FP's $i$th move is not a threat. Suppose for a contradiction that this is not the case. Assume first that, for every $m-1 \leq i < s$, FP's $i$th move is a standard threat. We will prove by induction on $i$ that, for every $m-1 \leq i < s$, immediately after FP's $i$th move, ${\mathcal G}_1$ satisfies the following three properties:
\begin{description}
\item [(a)] ${\mathcal G}_1$ contains a unique copy ${\mathcal F}_1$ of $\mathcal{H} \setminus \{z\}$;
\item [(b)] Let $e_{m-1}, \ldots, e_i$ denote the edges of ${\mathcal G}_1 \setminus {\mathcal F}_1$. Then, for every $m-1 \leq j \leq i$, there exists a vertex $z_j \in V({\mathcal G}_1)$ such that $\{z_j\} = e_j \setminus V({\mathcal F}_1)$ and $d_{{\mathcal G}_1}(z_j) = 1$;
\item [(c)] ${\mathcal F}_1 \cup \{e_i\}$ is open and ${\mathcal F}_1 \cup \{e_j\}$ is closed for every $m-1 \leq j < i$.  
\end{description}

Properties (a), (b) and (c) hold by assumption for $i = m-1$. Assume they hold for some $i \geq m-1$; we will prove they hold for $i+1$ as well. Immediately after FP's $i$th move, let $e'_i$ be a free edge such that ${\mathcal F}_1 \cup \{e_i, e'_i\} \cong {\mathcal H}$. Note that $e'_i$ exists by Property (c) and is unique by Lemma~\ref{lem::uniquerg}. According to his strategy, SP claims $e'_i$ thus closing ${\mathcal F}_1 \cup \{e_i\}$. By assumption, in his $(i+1)$th move FP makes a standard threat by claiming an edge $e_{i + 1}$. It follows that $e_{i + 1} \setminus V({\mathcal F}_1) = \{z_{i + 1}\}$, where, immediately after FP's $(i+1)$th move, $d_{{\mathcal G}_1}(z_{i + 1}) = 1$. Hence, Property (b) is satisfied immediately after FP's $(i+1)$th move. Since $\delta({\mathcal H} \setminus \{z\}) \geq 3$ holds by Property (ii), it follows that Property (a) is satisfied as well. Finally, ${\mathcal G}_1$ satisfies Property (c) by Lemma~\ref{lem::uniquerg}. Now, by Properties (a), (b) and (c), for every $m-1 \leq i < s$, immediately after FP's $i$th move there is a unique threat $e'_i$. According to his strategy, SP claims $e'_i$ in his $i$th and thus FP cannot win the game in his $(i+1)$th move. In particular, FP cannot win the game in his $s$th move, contrary to our assumption.

Assume then that there exists some $m \leq i < s$ such that FP makes a special threat in his $i$th move. We will prove that this is not possible. Consider the smallest such $i$. As discussed in the previous paragraph, immediately before FP's $i$th move, ${\mathcal G}_1$ contained a unique copy ${\mathcal F}_1$ of ${\mathcal H} \setminus \{z\}$, and every vertex of ${\mathcal G}_1 \setminus {\mathcal F}_1$ had degree one in ${\mathcal G}_1$. If FP makes a special threat in his $i$th move by claiming some edge $f'_1$, then there exists a free edge $f'_2$ such that, by claiming $f'_2$ in his $(i+1)$th move, FP would complete a copy ${\mathcal H}_1$ of ${\mathcal H}$. Since $|V({\mathcal F}_1)| < |V({\mathcal H})|$, there is some vertex $u \in V({\mathcal H}_1) \setminus V({\mathcal F}_1)$. Immediately after FP's $(i+1)$th move, the degree of $u$ in ${\mathcal G}_1$ is at most three. Hence, by Property (ii), $u$ must play the role of $z$ in ${\mathcal H}_1$. Therefore, ${\mathcal H}_1 = ({\mathcal F}_1 \cup \{f'_1, f'_2, e'_u\}) \setminus \{f'_3\}$, where $e'_u$ is the first edge incident with $u$ which FP has claimed and $f'_3$ is some edge of ${\mathcal F}_1$. Since, at some point in the game, $e'_u$ was a standard threat, and, at that point, ${\mathcal F}_1$ was the unique copy of ${\mathcal H} \setminus \{z\}$ in ${\mathcal G}_1$, there exists an edge $e''_u$ such that ${\mathcal H}' := {\mathcal F}_1 \cup \{e'_u, e''_u\} \cong {\mathcal H}$. Let $\phi : V({\mathcal H}') \rightarrow V(\mathcal{H}_1)$ be an isomorphism. It is evident that ${\mathcal H}' \setminus \{e''_u, f'_3\} = {\mathcal H}_1 \setminus \{f'_1, f'_2\}$ and that the restriction of $\phi$ to $V({\mathcal H}' \setminus \{e''_u, f'_3\})$ is a monomorphism and is thus the identity by Property (iv). However, $V({\mathcal H}' \setminus \{e''_u, f'_3\}) = V({\mathcal H}') = V({\mathcal F}_1) \cup \{u\} = V({\mathcal H}_1)$ and thus $\phi$ itself is the identity entailing $e''_u \in {\mathcal G}_1$. However, $e''_u \in {\mathcal G}_2$ holds by the description of the proposed strategy. Hence $e''_u \in {\mathcal G}_1 \cap {\mathcal G}_2$ which is clearly a contradiction.    

We conclude that there must exist some $m \leq i < s$ such that FP's $i$th move is not a threat. Let $\ell$ denote the first such move. In order to complete the proof of the lemma, we will prove by induction on $i$ that the following two properties hold immediately after FP's $i$th move for every $i \geq \ell$. 
\begin{description}
\item [(1)] FP has no threat.
\item [(2)] Let ${\mathcal G}'_1 = {\mathcal F}_1 \cup \{f\}$, where ${\mathcal F}_1$ is the unique copy of ${\mathcal H} \setminus \{z\}$ FP has built during his first $m-1$ moves and $f$ is the edge FP has claimed in his $\ell$th move. Then ${\mathcal G}_1 \setminus {\mathcal G}'_1$ consists of $i-m+1$ edges $e_m, \ldots, e_i$, where, for every $m \leq j \leq i$, $e_j$ contains a vertex $z_j$ such that $d_{{\mathcal G}_1}(z_j) = 1$.  
\end{description}

Properties (1) and (2) hold for $i = \ell$ by assumption, by the choice of $\ell$ and by Properties (a) -- (c) above. Assume they hold for some $i \geq \ell$; we will prove they hold for $i+1$ as well. Proving Property (2) can be done by essentially the same argument as the one used to prove Property (b) in Lemma~\ref{lem::noThreat}; the details are therefore omitted. Suppose for a contradiction that Property (1) does not hold immediately after FP's $(i+1)$th move. As in the proof of Property (a) in Lemma~\ref{lem::noThreat}, it follows that there are free edges $r''$ and $g''$ and graphs ${\mathcal F}^r \subseteq {\mathcal G}'_1$ and ${\mathcal F}^g \subseteq {\mathcal G}'_1$ such that ${\mathcal F}^r \cup \{r', g''\} \cong \mathcal{H} \cong {\mathcal F}^g \cup \{r'', g'\}$. Since ${\mathcal F}^r \subseteq {\mathcal G}'_1$ and ${\mathcal F}^g \subseteq {\mathcal G}'_1$, it follows by Property (ii) that $V({\mathcal F}^r) = V({\mathcal F}^g)$. Let ${\mathcal F}_2 \subseteq {\mathcal G}_2$ be such that ${\mathcal F}_2 \cup \{r', g'\} \cong \mathcal{H}$ and let $z'$ be the unique vertex in $r' \setminus V({\mathcal F}_2)$. Note that $r' \setminus \{z'\} \subseteq V({\mathcal F}^r)$ and $g' \setminus \{z'\} \subseteq V({\mathcal F}^g)$. Hence $(r' \cup g') \setminus \{z'\} \subseteq V({\mathcal F}^r)$. By Property (vi), we then have $|V({\mathcal F}_2) \setminus V({\mathcal F}^r)| \leq |V({\mathcal F}_2) \setminus (r' \cup g')| < k-1$. However, $e_1 \cap V({\mathcal F}_2) = \emptyset$ holds by the description of the proposed strategy and $|e_1 \cap V({\mathcal F}^r)| \geq k-1$ holds by our assumption that FP's $(m-1)$th move was a threat. This implies that $k-1 \leq |e_1 \cap V({\mathcal F}^r)| \leq |V({\mathcal F}^r) \setminus V({\mathcal F}_2)| = |V({\mathcal F}_2) \setminus V({\mathcal F}^r)| < k-1$ which is clearly a contradiction. We conclude that Property (1) does hold immediately after FP's $(i+1)$th move.      
\end{proof}

Since FP's $(m-1)$th move is either a standard threat or a special threat or no threat at all, Theorem~\ref{th::HypergraphProperties} follows immediately from Lemmata~\ref{lem::noThreat}, \ref{lem::specialThreat} and~\ref{lem::standardThreat}. 
{\hfill $\Box$ \medskip\\}

\section{An explicit construction} \label{sec::example}
\newcommand{\e}[1]
{
  \ifnum\pdfstrcmp{#1}{r}=0
   r
  \else
  \ifnum\pdfstrcmp{#1}{g}=0
   g
  \else
  \ifnum\pdfstrcmp{#1}{a}=0
   a
  \else
  \ifnum\pdfstrcmp{#1}{9,4}=0
   b  
  \else
  \ifnum\pdfstrcmp{#1}{4,9}=0
   b  
  \else
  \ifnum\pdfstrcmp{#1}{1,5}=0
   e_1  
  \else
  \ifnum\pdfstrcmp{#1}{2,6}=0
   e_2  
  \else
  \ifnum\pdfstrcmp{#1}{3,7}=0
   e_3 
  \else
  \ifnum\pdfstrcmp{#1}{4,8}=0
   e_4 
  \else
  \ifnum\pdfstrcmp{#1}{5,9}=0
   e_5 
  \else
  \ifnum\pdfstrcmp{#1}{9,5}=0
   e_5  
  \else
   e_{?#1?}
  \fi\fi\fi  \fi\fi\fi  \fi\fi\fi  \fi\fi
}

\newcommand{\V}[1]
{
  \ifnum\pdfstrcmp{#1}{z}=0
    z  
  \else
   v_{#1}
  \fi
}

\newcommand{\Hef}{\mathcal H_{ef}}

In this section we will describe a $5$-graph ${\mathcal H}$ which satisfies Properties (i) -- (vi) from Theorem~\ref{th::HypergraphProperties} and thus $\mathcal{R}^{(5)}(\mathcal{H}, \aleph_0)$ is a draw. The vertex set of $\mathcal{H}$ is $\{\V{z}, \V{1},\V{2},\V{3},\V{4},\V{5},\V{6},\V{7},\V{8},\V{9}\}$, and its edges are 
\begin{align*}
\e{r}&=\{\V{z}, \V{1}, \V{3}, \V{5}, \V{8}\},\\
\e{g}&=\{\V{z}, \V{2}, \V{4}, \V{7}, \V{9}\},\\
\e{a}&=\{\V{1}, \V{4}, \V{6}, \V{8}, \V{9}\},\\
\e{9,4}&=\{\V{9}, \V{1}, \V{2}, \V{3}, \V{4}\},\\
\e{1,5}&=\{\V{1}, \V{2}, \V{3}, \V{4}, \V{5}\},\\
\e{2,6}&=\{\V{2}, \V{3}, \V{4}, \V{5}, \V{6}\},\\
\e{3,7}&=\{\V{3}, \V{4}, \V{5}, \V{6}, \V{7}\},\\
\e{4,8}&=\{\V{4}, \V{5}, \V{6}, \V{7}, \V{8}\},\\
\e{5,9}&=\{\V{5}, \V{6}, \V{7}, \V{8}, \V{9}\}.
\end{align*}

\begin{figure} \label{fig::H5}
\centering
\includegraphics[scale=1]{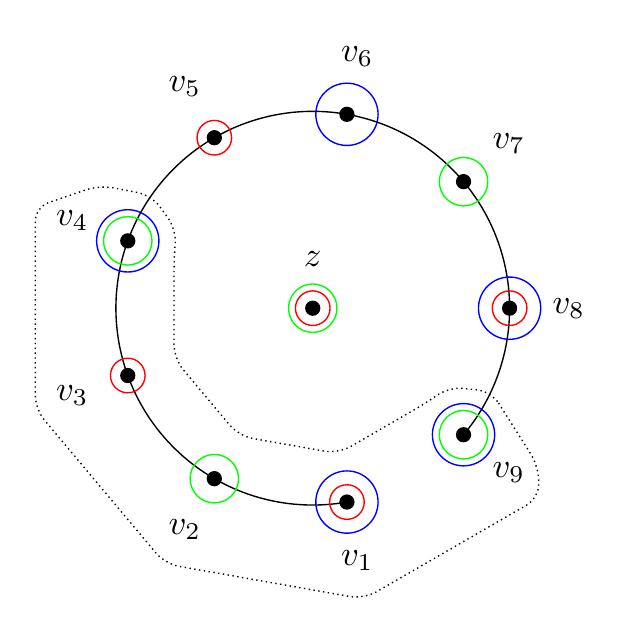}
\caption{The $5$-uniform hypergraph $\mathcal{H}$. The black line from $v_1$ to $v_9$ represents the tight path consisting of the edges $\e{1,5}, \dots, \e{5,9}$.}
\end{figure}

It readily follows from the definition of $\mathcal{H}$ that it satisfies Properties (i), (ii), (v) and (vi) from Theorem~\ref{th::HypergraphProperties}. We claim that it satisfies Properties (iii) and (iv) as well. We start with Property (iii).

\begin{lemma} \label{lem::property3}
${\mathcal H} \setminus \{\V{z}\}$ has a fast winning strategy.
\end{lemma}

\begin{proof}
We describe a strategy for SP to build a copy of ${\mathcal H} \setminus \{\V{z}\}$ in seven moves. The basic idea is to build a tight path of length $5$ in five moves, and then to use certain symmetries of ${\mathcal H} \setminus \{\V{z}\}$ in order to complete a copy of ${\mathcal H} \setminus \{\V{z}\}$ in two additional moves. Our strategy is divided into the following three stages.

\bigskip

\noindent \textbf{Stage I:} In his first move, SP claims an arbitrary free edge $e_1 = \{v_1, v_2, v_3, v_4, v_5\}$. For every $2 \leq i \leq 5$, in his $i$th move SP picks a vertex $v_{i+4}$ which is isolated in both his and FP's current graphs and claims the edge $e_i = \{v_i, v_{i+1}, v_{i+2}, v_{i+3}, v_{i+4}\}$. If in his $6$th move FP claims either $\{v_1, v_2, v_3, v_4, v_9\}$ or $\{v_1, v_4, v_6, v_8, v_9\}$ or $\{v_1, v_3, v_5, v_8, v_9\}$, then SP claims $\{v_1, v_6, v_7, v_8, v_9\}$ and proceeds to Stage II. Otherwise, SP claims $\{v_1, v_2, v_3, v_4, v_9\}$ and skips to Stage III. 

\medskip

\noindent \textbf{Stage II:} If in his seventh move FP claims $\{v_1, v_2, v_4, v_6, v_9\}$, then SP claims $\{v_1, v_2, v_5, v_7, v_9\}$. Otherwise, SP claims $\{v_1, v_2, v_4, v_6, v_9\}$.  

\medskip

\noindent \textbf{Stage III:} If in his seventh move FP claims $\{v_1, v_4, v_6, v_8, v_9\}$, then SP claims $\{v_1, v_3, v_5, v_8, v_9\}$. Otherwise, SP claims $\{v_1, v_4, v_6, v_8, v_9\}$.	

\bigskip

It is easy to see that SP can indeed play according to the proposed strategy and that, in each of the possible cases, the graph he builds is isomorphic to ${\mathcal H} \setminus \{z\}$. 
\end{proof}

It remains to prove that ${\mathcal H}$ satisfies Property (iv). We begin by introducing some additional notation. If $\phi : V(\mathcal{H}) \rightarrow V(\mathcal{H})$ is a monomorphism, and $e = \{a_1, a_2, \ldots, a_5\} \in \mathcal{H}$, then we set $\phi(e) := \{\phi(a_1), \phi(a_2), \ldots, \phi(a_5)\}$. For two edges $e, f \in \mathcal{H}$, let $\mathcal{H}_{ef} = \mathcal{H} \setminus \{e, f\}$.  

Next, we observe several simple properties of $\mathcal{H}$ and of monomorphisms. Table~\ref{DegreeTable} shows the degrees of the vertices in $\mathcal H$ and Table~\ref{IntersectionTable} shows the sizes of intersections of pairs of edges in $\mathcal H$.
\begin{table}[h]
\begin{center}
\begin{tabular}{|l*{10}{|p{0.6cm}}|}
 \hline
    Vertex    &$\V{1}$&$\V{2}$&$\V{3}$&$\V{4}$&$\V{5}$&$\V{6}$&$\V{7}$&$\V{8}$&$\V{9}$&$\V{z}$ \\ \hline
	Degree    &  4    &   4   &  5    &   7   &    6  &    5  &    4  &    4  &  4    &   2   \\ \hline
\end{tabular}
\end{center}
\caption{Degrees of vertices in $\mathcal{H}$.}\label{DegreeTable}
\end{table} 

\begin{table}[h]
\begin{center}
\begin{tabular}{|l*{9}{|p{0.6cm}}|}
 \hline
              &$\e{r}$&$\e{g}$&$\e{a}$&$\e{9,4}$&$\e{1,5}$ &$\e{2,6}$&$\e{3,7}$&$\e{4,8}$&$\e{5,9}$ \\ \hline
	$\e{r}$     &       &   1   &  2    &    2    &    3     &    2    &    2    &    2    &    2    \\ \hline
	$\e{g}$     &   1   &       &  2    &    3    &    2     &    2    &    2    &    2    &    2    \\ \hline
	$\e{a}$     &   2   &   2   &       &    3    &    2     &    2    &    2    &    3    &    3    \\ \hline
	$\e{9,4}$   &   2   &   3   &  3    &         &    4     &    3    &    2    &    1    &    1    \\ \hline
	$\e{1,5}$   &   3   &   2   &  2    &    4    &          &    4    &    3    &    2    &    1    \\ \hline
	$\e{2,6}$   &   2   &   2   &  2    &    3    &    4     &         &    4    &    3    &    2    \\ \hline
	$\e{3,7}$   &   2   &   2   &  2    &    2    &    3     &    4    &         &    4    &    3    \\ \hline
	$\e{4,8}$   &   2   &   2   &  3    &    1    &    2     &    3    &    4    &         &    4    \\ \hline
	$\e{5,9}$   &   2   &   2   &  3    &    1    &    1     &    2    &    3    &    4    &         \\ \hline
\end{tabular}
\caption{Intersection sizes of pairs of edges in $\mathcal H$.}\label{IntersectionTable}
\end{center}
\end{table}

\begin{observation} \label{obs::H5properties}
The hypergraph $\mathcal{H}$ satisfies all of the following properties:
\begin{description}
\item[(1)] $V(\mathcal{H}) \setminus \{\e{r}, \e{g}\} = \{\V6\}$ and $\e{r} \cap \e{g} = \{\V{z}\}$.
\item[(2)] $\e{1,5}$ is the unique edge satisfying $|\e{1,5} \cap \e{r}| = 3$ and $|\e{1,5} \cap \e{g}| = 2$.
\item[(3)] $\e{4,9}$ is the unique edge satisfying $|\e{4,9} \cap \e{g}| = 3$ and $|\e{4,9} \cap \e{r}| = 2$.
\item[(4)] There are precisely two tight paths of length five in $\mathcal{H}$, namely, $TP_1 := (\e{1,5}, \e{2,6}, \e{3,7}, \e{4,8}, \e{5,9})$ and $TP_2 := (\e{9,4}, \e{1,5}, \e{2,6}, \e{3,7}, \e{4,8})$.    
\item[(5)] For every two vertices $u, v \in V(\mathcal{H})$, there are three edges $f_1, f_2, f_3 \in \mathcal{H}$ such that $|f_i \cap \{u,v\}| = 1$ for every $1 \leq i \leq 3$.   
\end{description}
\end{observation}

\begin{observation} \label{obs::monomorphisms}
Let $\mathcal{F}$ and $\mathcal{F}'$ be $k$-graphs, where $\mathcal{F}' \subseteq \mathcal{F}$, and let $\phi : V(\mathcal{F}') \rightarrow V(\mathcal{F})$ be a monomorphism. Then
\begin{description}
\item[(a)] $d_{\mathcal F}(\phi(x)) \geq d_{\mathcal{F}'}(x)\geq d_{\mathcal F}(\phi(x))-|E(\mathcal F \setminus \mathcal F')|$ holds for every $x \in V(\mathcal{F}')$. 
\item[(b)] If $P$ is a tight path of length $\ell$ in $\mathcal{F}'$, then $\phi(P)$ is a tight path of length $\ell$ in $\mathcal{F}$.   \item[(c)] Let $P = (f_1, f_2, \ldots, f_m)$ be a tight path in $\mathcal{F}'$, where $m \geq k$ and $f_i = \{p_i, \ldots, p_{i + k - 1}\}$ for every $1 \leq i \leq m$. If $\phi(P) = (e_1, e_2, \ldots, e_m)$, where $e_i = \{q_i, \ldots, q_{i + k - 1}\}$ for every $1 \leq i \leq m$, then either $\phi(p_i) = q_i$ for every $1 \leq i \leq m+k-1$ or $\phi(p_i) = q_{m+k-i}$ for every $1 \leq i \leq m+k-1$.    
\item[(d)] For any pair of edges $x, y \in \mathcal F'$ we have $|\phi(x) \cap \phi(y)| = |x \cap y|$.
\end{description}
\end{observation}

We prove that ${\mathcal H}$ satisfies Property (iv) in a sequence of lemmata.

\begin{lemma} \label{3edges}
Let $e$ and $f$ be two arbitrary edges of ${\mathcal H}$ and let $\phi : V(\mathcal{H}_{ef}) \rightarrow V(\mathcal{H})$ be a monomorphism. If $\phi(e') = e'$ holds for every edge $e' \in \mathcal{H}_{ef}$, then $\phi$ is the identity.
\end{lemma}

\begin{proof}
Suppose for a contradiction that $\phi$ is not the identity. Then, there exist distinct vertices $u, v \in V(\mathcal{H}_{ef})$ such that $\phi(u) = v$. By Observation~\ref{obs::H5properties}(5), there are three edges $f_1, f_2, f_3 \in {\mathcal H}$ such that $|f_i \cap \{u, v\}| = 1 $ for every $1 \leq i \leq 3$. Clearly, we may assume that $f_1 \notin \{e, f\}$ and thus $\phi(f_1) = f_1$ by the assumption of the lemma. Since $\phi(u) = v$, it follows that $\{u, v\} \subseteq f_1$ which is a contradiction.
\end{proof}

\begin{lemma} \label{zFixed}
Let $\phi : V(\mathcal{H}_{ef}) \rightarrow V(\mathcal{H})$ be a monomorphism. Then $\phi(\V{z}) = \V{z}$.
\end{lemma}  

\begin{proof}
Assume first that $\{e,f\} \cap \{\e{r}, \e{g}\} \neq \emptyset$. Then $d_{\Hef}(\V{z}) \leq 1$. Combined with  Observation~\ref{obs::monomorphisms}(a), this implies that $d_{\mathcal H}(\phi(\V{z})) \leq 1 + |\{e, f\}| = 3$. Since $z$ is the only vertex of degree at most $3$ in $\mathcal H$, it follows that $\phi(\V{z}) = \V{z}$.

Assume then that $\{e,f\} \cap \{\e{r}, \e{g}\} = \emptyset$. Since $\phi$ is a monomorphism, there exists a vertex $v \in V(\mathcal H_{ef})$ such that $\phi(v) = \V{z}$. Suppose for a contradiction that $v \neq \V{z}$. By Observation~\ref{obs::monomorphisms}(a), we have $d_{\Hef}(v) \leq 2$ and thus $d_{\mathcal H}(v) \leq 4$. Since $\V{z}$ is the only vertex of degree less than $4$ in $\mathcal H$, it follows that $d_{\mathcal H}(v) = 4$ and that both $e$ and $f$ contain $v$. Let $r' = \phi^{-1}(\e{r})$ and $g' = \phi^{-1}(\e{g})$ be the other two edges of $\mathcal{H}$ that contain $v$. By Observation~\ref{obs::monomorphisms}(d), we have $|r' \cap g'| = |\e{r} \cap \e{g}| = 1$. Looking at Tables~\ref{DegreeTable} and~\ref{IntersectionTable}, we see that the only choice of $r', g'$ and $v$ such that $d_{\mathcal H}(v) = 4$ and $r' \cap g' = \{v\}$ is $v = \V{9}$ and $\{r', g'\} = \{\e{9,4}, \e{5,9}\}$. Since both $e$ and $f$ contain $v$ as well, this implies that $\{e, f\} = \{\e{g}, \e{a}\}$, contrary to our assumption that $\{e,f\} \cap \{\e{r}, \e{g}\} = \emptyset$.
\end{proof}

\begin{lemma} \label{rg}
Let $\phi : V(\mathcal{H}_{ef}) \rightarrow V(\mathcal{H})$ be a monomorphism.  If $\e{r}, \e{g} \in \mathcal {H}_{ef}$, $\phi(\e{r}) = \e{r}$ and $\phi(\e{g}) = \e{g}$, then $\phi$ is the identity.
\end{lemma}   

\begin{proof}
Since $\phi$ is injective, $\phi(\e{r}) = \e{r}$, and $\phi(\e{g}) = \e{g}$, it follows by Observation~\ref{obs::H5properties} (1) that $\phi(\V{6}) = \V{6}$.

By Observation~\ref{obs::monomorphisms}(a), we have that $d_{\mathcal H}(\phi(\V{4})) \geq d_{\Hef}(\V{4}) \geq 5$ which in turn implies that $\phi(\V{4}) \in \{\V{3}, \V{4}, \V{5}, \V{6}\}$. Since, moreover, $\phi(\V{4}) \in \phi(\e{g}) = \e{g} = \{\V{z}, \V{2}, \V{4}, \V{7}, \V{9}\}$, it follows that $\phi(\V{4}) = \V{4}$. Since $\e{5,9}$ is the unique edge in $\mathcal H$ containing $\V{6}$ but not $\V{4}$, we have that if $\e{5,9} \in \Hef$, then $\phi(\e{5,9}) = \e{5,9}$.

Since $\phi(\e{r}) = \e{r}$ and $\phi(\e{g}) = \e{g}$, it follows by Observation~\ref{obs::monomorphisms}(d) and by Observation~\ref{obs::H5properties}(2), that if $\e{1,5} \in \Hef$, then $\phi(\e{1,5})=\e{1,5}$. Similarly, using Observation~\ref{obs::H5properties}(3), it follows that if $\e{9,4} \in \Hef$, then $\phi(\e{9,4}) = \e{9,4}$. We distinguish between the following three cases.

\textbf{Case 1: $\e{9,4}, \e{3,7} \in \Hef$}. As noted above $\phi(\e{9,4}) = \e{9,4}$. Since $|\e{3,7} \cap \e{9,4}| = 2$, Observation~\ref{obs::monomorphisms}(d) and Table~\ref{IntersectionTable} imply that $\phi(\e{3,7}) \in \{\e{3,7}, \e{r}\}$. Since, moreover, $\phi(\e{r}) = \e{r}$ by assumption, we conclude that $\phi(\e{3,7}) = \e{3,7}$. Observation~\ref{obs::monomorphisms}(d) then  implies that $(|\phi(x) \cap \e{9,4}|, |\phi(x) \cap \e{3,7}|) = (|x \cap \e{9,4}|, |x \cap \e{3,7}|)$ for every edge $x \in \Hef$. Looking at the rows corresponding to $\e{9,4}$ and $\e{3,7}$ in Table~\ref{IntersectionTable}, we see that the pair $(|x \cap \e{9,4}|, |x \cap \e{3,7}|)$ is distinct for every $x \in \mathcal H \setminus \{\e{r}, \e{g}\}$. It follows that $\phi(x) = x$ for every $x \in \Hef$. Hence, $\phi$ is the identity by Lemma~\ref{3edges}.

\textbf{Case 2: $\e{2,6}, \e{5,9} \in \Hef$}. As noted above $\phi(\e{5,9}) = \e{5,9}$. Since $|\e{2,6} \cap \e{5,9}| = 2$, Observation~\ref{obs::monomorphisms}(d) and Table~\ref{IntersectionTable} imply that $\phi(\e{2,6}) \in \{\e{2,6}, \e{r}, \e{g}\}$. Since, moreover, $\phi(\e{r}) = \e{r}$ and $\phi(\e{g}) = \e{g}$ by assumption, we conclude that $\phi(\e{2,6}) = \e{2,6}$. Observation~\ref{obs::monomorphisms}(d) then implies that $(|\phi(x) \cap \e{5,9}|, |\phi(x) \cap \e{2,6}|) = (|x \cap \e{5,9}|, |x \cap \e{2,6}|)$ for every edge $x \in \Hef$. Looking at the rows corresponding to $\e{5,9}$ and $\e{2,6}$ in Table~\ref{IntersectionTable}, we see that the pair $(|x \cap \e{5,9}|, |x \cap \e{2,6}|)$ is distinct for every $x \in \mathcal H \setminus \{\e{r}, \e{g}\}$. It follows that $\phi(x) = x$ for every $x \in \Hef$. Hence, $\phi$ is the identity by Lemma~\ref{3edges}.

\textbf{Case 3: $\{e, f\} \in \{\e{9,4}, \e{3,7}\} \times \{\e{2,6}, \e{5,9}\}$}. Observe that $\e{1,5} \in \Hef$ and thus, as noted above, $\phi(\e{1,5}) = \e{1,5}$. Looking at the row corresponding to $\e{1,5}$ in Table~\ref{IntersectionTable} and using Observation~\ref{obs::monomorphisms}(d), we infer that $\phi(\e{3,7}) = \e{3,7}$, $\phi(\e{5,9}) = \e{5,9}$, $\{\phi(\e{9,4}), \phi(\e{2,6})\} = \{\e{9,4}, \e{2,6}\}$, and $\{\phi(\e{a}), \phi(\e{4,8})\} = \{\e{a}, \e{4,8}\}$. Since $\phi(\V{6}) = \V{6}$, it then follows that $\phi(\e{2,6}) = \e{2,6}$ and thus $\phi(\e{9,4}) = \e{9,4}$. Let $x$ denote the unique edge of $\{\e{2,6}, \e{5,9}\} \cap \Hef$. Looking at the row corresponding to $x$ in Table~\ref{IntersectionTable}, we see that $|x \cap \e{a}| \neq |x \cap \e{4,8}|$. Using Observation~\ref{obs::monomorphisms}(d), we conclude that $\phi(\e{a}) = \e{a}$ and $\phi(\e{4,8}) = \e{4,8}$. Hence, $\phi$ is the identity by Lemma~\ref{3edges}.     

Since, clearly, at least one of the above three cases must occur, this concludes the proof of the lemma.
\end{proof}

\begin{lemma} \label{FixV9}
Let $\phi : V(\mathcal{H}_{ef}) \rightarrow V(\mathcal{H})$ be a monomorphism. If $\e{9,4}, \e{5,9} \in \mathcal H_{ef}$, then $\phi(\V{9}) = \V{9}$.
\end{lemma}   

\begin{proof}
Suppose for a contradiction that $\phi(\V{9}) \neq \V{9}$. By Lemma~\ref{zFixed} we have $\phi(\V{z})=\V{z}$ which implies that $\phi(\V{9})\neq \V{z}$. By Observation~\ref{obs::monomorphisms}(a) we have $d_{\mathcal H}(\phi(\V{9}))\leq d_{\Hef}(\V{9}) + 2 \leq 6$ which implies that $\phi(\V{9}) \neq \V{4}$. 

Since $\phi$ is a monomorphism, we have $\{\phi(\V{9})\} = \phi(\e{9,4}\cap \e{5,9}) = \phi(\e{9,4}) \cap \phi(\e{5,9})$. Since $\{\phi(\V{9})\}$ is the intersection of two edges, we must have $\phi(\V{9}) \in \{\V{4}, \V{5}, \V{9}, \V{z}\}$. Combining this with  the previous paragraph, we infer that $\phi(\V{9}) = \V{5}$.    

Note that $6 = d_{\mathcal H}(\V{5}) = d_{\mathcal H}(\phi(\V{9})) \leq d_{\Hef}(\V{9}) + 2$. Hence, $d_{\Hef}(\V{9}) = 4$ which implies that $\{\e{g}, \e{a}\} \cap \{e, f\} = \emptyset$. Since $\phi(\V{z}) = \V{z}$ and $\phi(\V{9}) = \V{5}$, we must have $\phi(\e{g}) = \e{r}$.

Since $\e{1,5}, \e{5,9}$ is the unique pair of edges satisfying $\e{1,5} \cap \e{5,9} = \{\V{5}\}$, it follows that $\{\phi(\e{9,4}), \phi(\e{5,9})\} = \{\e{1,5}, \e{5,9}\}$. Suppose for a contradiction that $\phi(\e{9,4}) = \e{5,9}$. Then, by Observation~\ref{obs::monomorphisms}(d) we have $3 = |\e{9,4} \cap \e{g}| = |\phi(\e{9,4}) \cap \phi(\e{g})| = |\e{5,9} \cap \e{r}| = 2$. We conclude that $\phi(\e{9,4}) = \e{1,5}$ and $\phi(\e{5,9}) = \e{5,9}$. We can now determine the missing edges in $\Hef$ and in $\phi(\Hef)$.

\begin{claim}\label{Claime26e48}
$\e{2,6}, \e{4,8} \not \in \Hef$ and $\e{2,6}, \e{4,8} \not \in \phi(\Hef)$.
\end{claim}

\begin{proof}
Suppose for a contradiction that $\e{2,6} \in \Hef$. Since $|\e{2,6} \cap \e{4,9}| = 3$ and $\V{9} \notin \e{2,6}$, it follows by Observation~\ref{obs::monomorphisms}(d) that $|\phi(\e{2,6}) \cap \e{1,5}| = |\phi(\e{2,6}) \cap \phi(\e{4,9})| = 3$ and $\V{5} \notin \phi(\e{2,6})$. This is a contradiction since there is no edge $x \in \mathcal{H}$ such that $|x \cap \e{1,5}| = 3$ and $\V{5} \notin x$.

Suppose for a contradiction that $\e{4,8} \in \Hef$. It follows by Observation~\ref{obs::monomorphisms}(d) that $4 = |\e{4,8} \cap \e{5,9}| = |\phi(\e{4,8}) \cap \phi(\e{5,9})| = |\phi(\e{4,8}) \cap \e{5,9}|$ and thus $\phi(\e{4,8}) = \e{4,8}$. Since, moreover, $\V9 \notin \e{4,8}$ and $\phi(\V{9}) = \V{5}$, it follows that $\V5 \notin \e{4,8}$, contrary to the definition of $\e{4,8}$.    

Suppose for a contradiction that $\e{2,6} \in \phi(\Hef)$. Let $x \in \Hef$ be such that $\phi(x) = \e{2,6}$. Since $\phi(\e{9,4}) = \e{1,5}$, it follows by Observation~\ref{obs::monomorphisms}(d) that $4 = |\e{1,5} \cap \e{2,6}| = |\e{9,4} \cap x|$. Looking at the row corresponding to $\e{9,4}$ in Table~\ref{IntersectionTable}, we infer that $x = \e{1,5}$. However, since $\V9 \notin \e{1,5}$, we then deduce that $\V5 = \phi(\V9) \notin \e{2,6}$ which is clearly a contradiction.  

Suppose for a contradiction that $\e{4,8} \in \phi(\Hef)$. Let $x \in \Hef$ be such that $\phi(x) = \e{4,8}$. Since $\phi(\e{5,9}) = \e{5,9}$, it follows by Observation~\ref{obs::monomorphisms}(d) that $4 = |\e{4,8} \cap \e{5,9}| = |x \cap \e{5,9}|$. Looking at the row corresponding to $\e{5,9}$ in Table~\ref{IntersectionTable}, we infer that $x = \e{4,8}$. However, we already saw before that assuming $\e{4,8} \in \Hef$ results in a contradiction. 
\end{proof}  

We are now in a position to complete the proof of Lemma~\ref{FixV9}. Let $\mathcal{F} = \mathcal{H} \setminus \{\e{2,6}, \e{4,8}\}$. It follows from Claim~\ref{Claime26e48} that $\Hef = \phi(\Hef) = \mathcal{F}$ and that $\phi$ is an automorphism of $\mathcal{F}$. Hence, in particular, $d_{\mathcal{F}}(\phi(\V4)) = d_{\mathcal{F}}(\V4) = 5$. On the other hand, since $\phi(\e{g}) = \e{r}$, it follows that $\phi(v_4) \in \{\V1, \V3, \V5, \V8\}$. Therefore $d_{\mathcal{F}}(\phi(\V4)) \leq 4$ which is clearly a contradiction.      
\end{proof}

\begin{lemma}\label{TightPathPresent}
Let $\phi : V(\mathcal{H}_{ef}) \rightarrow V(\mathcal{H})$ be a monomorphism. Suppose that $\Hef$ contains a tight path of length $5$. Then $\phi$ is either the identity or one of
$(\V9\V1\V2\V3\V4\V5\V6\V7\V8)(\V{z})$, 
$(\V1\V9\V8\V7\V6\V5\V4\V3\V2)(\V{z})$,
$(\V9\V8)(\V1\V7)(\V2\V6)(\V3\V5)(\V4)(\V{z})$,
$(\V1\V8)(\V2\V7)(\V3\V6)(\V4\V5)(\V9)(\V{z})$, and
$(\V1\V9)(\V2\V8)(\V3\V7)(\V4\V6)(\V5)(\V{z})$.
\end{lemma}

\begin{proof}
By Lemma~\ref{zFixed} we know that $\phi(\V{z}) = \V{z}$. Moreover, by Observation~\ref{obs::H5properties}(4), we know that $\Hef$ contains $TP1$ or $TP2$. Moreover, by Observation~\ref{obs::monomorphisms}(b), if $TP1 \in \Hef$, then $\phi(TP1) \in \{TP1, TP2\}$ and if $TP2 \in \Hef$, then $\phi(TP2) \in \{TP1, TP2\}$. Accordingly, we distinguish between the following four cases.

\textbf{Case 1: $TP1 \in \Hef$ and $\phi(TP1) = TP1$}. It follows by Observation~\ref{obs::monomorphisms}(c) that either $\phi$ is the identity or $\phi = (\V1\V9)(\V2\V8)(\V3\V7)(\V4\V6)(\V5)(\V{z})$.

\textbf{Case 2: $TP1 \in \Hef$ and $\phi(TP1) = TP2$}. It follows by Observation~\ref{obs::monomorphisms}(c) that either $\phi = (\V1\V9\V8\V7\V6\V5\V4\V3\V2)(\V{z})$ or $\phi = (\V1\V8)(\V2\V7)(\V3\V6)(\V4\V5)(\V9)(\V{z})$.

\textbf{Case 3: $TP2 \in \Hef$ and $\phi(TP2) = TP1$}. It follows by Observation~\ref{obs::monomorphisms}(c) that either $\phi = (\V9\V1\V2\V3\V4\V5\V6\V7\V8)(\V{z})$ or $\phi = (\V1\V8)(\V2\V7)(\V3\V6)(\V4\V5)(\V9)(\V{z})$.

\textbf{Case 4: $TP2 \in \Hef$ and $\phi(TP2) = TP2$}. It follows by Observation~\ref{obs::monomorphisms}(c) that either $\phi$ is the identity or $\phi = (\V8\V9)(\V1\V7)(\V2\V6)(\V3\V5)(\V4)(\V{z})$.
\end{proof}

\begin{lemma} \label{94and59}
Let $\phi : V(\mathcal{H}_{ef}) \rightarrow V(\mathcal{H})$ be a monomorphism. If $\e{9,4}, \e{5,9}\in \mathcal H_{ef}$, then $\phi$ is the identity.
\end{lemma}  

\begin{proof}
Suppose for a contradiction that $\phi$ is not the identity. By Lemma~\ref{zFixed} we know that $\phi(\V{z}) = \V{z}$ and by Lemma~\ref{FixV9} we know that $\phi(\V{9}) = \V{9}$. Assume first that $\phi(\e{4,9}) = \e{4,9}$. Since $\phi(\V{9}) = \V{9}$, $\phi(\e{4,9}) = \e{4,9}$, and $\e{5,9}$ is the unique edge whose intersection with $\e{4,9}$ is $\{\V{9}\}$, we infer that $\phi(\e{5,9}) = \e{5,9}$. Since $\e{g}$ is the unique edge containing both $\V9 $ and $ \V{z}$, we infer that, if $\e{g} \in \Hef$, then $\phi(\e{g}) = \e{g}$. Since $\e{2,6}$ is the unique edge satisfying $|\e{2,6} \cap \e{4,9}| = 3$, $|\e{2,6} \cap \e{5,9}| = 2$, and $|\e{2,6} \cap \e{4,9} \cap \e{5,9}| = 0$, it follows by Observation~\ref{obs::monomorphisms}(d) that, if $\e{2,6} \in \Hef$, then $\phi(\e{2,6})=\e{2,6}$. Looking at the rows corresponding to $\e{5,9}$ and $\e{4,9}$ in Table~\ref{IntersectionTable}, we see that $(|x \cap \e{5,9}|, |x \cap \e{4,9}|)$ is distinct for every $x \in \mathcal{H} \setminus\{\e{g}, \e{2,6}\}$. This implies that $\phi(x) = x$ for every $x \in \Hef$ and thus $\phi$ is the identity by Lemma~\ref{3edges} contrary to our assumption. Therefore, from now on we will assume that $\phi(\e{4,9}) \neq \e{4,9}$. Since $\phi(\V{9}) = \V{9}$, it follows that $\phi(\e{4,9}) = \e{5,9}$ and $\phi(\e{5,9}) = \e{4,9}$. We distinguish between the following three cases.

\textbf{Case 1: $\{e,f\} \subseteq \{\e{r}, \e{g}, \e{a}\}$}. Observe that $\Hef$ contains $TP1$. Since, moreover, $\phi(\V{9}) = \V{9}$ and $\phi$ is not the identity by assumption, it follows from Lemma~\ref{TightPathPresent} that $\phi = (\V1\V8)(\V2\V7)(\V3\V6)(\V4\V5)(\V9)(\V{z})$. Let $x \in \{\e{r}, \e{g}, \e{a}\} \setminus \{e,f\}$. Then $\phi(x)$ is not an edge of $\mathcal{H}$ contrary to $\phi$ being a monomorphism.    

\textbf{Case 2: $\e{g} \in \Hef$}. As noted above, $\phi(\e{g}) = \e{g}$. Since $\e{4,9}$ is the unique edge intersecting $\e{g}$ in $3$ vertices, we have $\phi(\e{4,9}) = \e{4,9}$, contrary to our assumption that $\phi(\e{4,9}) \neq \e{4,9}$. 

\textbf{Case 3: $\e{g} \notin \Hef$ and $\e{r}, \e{a} \in \Hef$}. Since $\e{r}$ is the unique edge such that $\V{z} \in \e{r}$ and $\V9 \notin \e{r}$, it follows that $\phi(\e{r}) = \e{r}$. Similarly, since $\V{z} \notin \e{a}$, $\V9 \in \e{a}$, $\phi(\e{4,9}) = \e{5,9}$, and $\phi(\e{5,9}) = \e{4,9}$, it follows that $\phi(\e{a}) = \e{a}$. Then 
\begin{align*}
\{\phi(\V1)\} &= \phi(\e{4,9}) \cap \phi(\e{r}) \cap \phi(\e{a}) = \e{5,9} \cap \e{r} \cap \e{a} = \{\V8\},\\
\{\phi(\V2)\} &= \phi(\e{4,9}) \setminus (\phi(\e{r}) \cup \phi(\e{a})) = \e{5,9} \setminus(\e{r} \cup \e{a}) = \{\V7\},\\
\{\phi(\V3)\} &= \phi(\e{4,9}) \cap \phi(\e{r}) \setminus \phi(\e{a}) = \e{5,9} \cap \e{r} \setminus \e{a} = \{\V5\}.
\end{align*}
Since, moreover, $\phi(\e{4,9}) = \e{5,9}$, it follows that $\phi(\V4) = \V6$. Now, using $\phi(\e{r}) = \e{r}$ and $\phi(\e{a}) = \e{a}$, it is easy to see that $\phi(\V8) = \V1$ and thus $\phi(\V6) = \V4$, $\phi(\V5) = \V3$ and $\phi(\V7) = \V2$. However, then neither $\phi(\e{1,5})$ nor $\phi(\e{4,8})$ is an edge of $\mathcal{H}$. Since $\{\e{1,5}, \e{4,8}\} \setminus \{e,f\} \neq \emptyset$, this contradicts $\phi$ being a monomorphism.       
\end{proof}

\begin{lemma} \label{94or59}
Let $\phi : V(\mathcal{H}_{ef}) \rightarrow V(\mathcal{H})$ be a monomorphism. If $\{e,f\} \in \{\e{9,4}, \e{5,9}\} \times \{\e{r}, \e{g}, \e{a}\}$, then $\phi$ is the identity.
\end{lemma}  

\begin{proof}
Since $|\{e,f\} \cap \{\e{9,4}, \e{5,9}\}| = 1$ by assumption, $\Hef$ contains either $TP1$ or $TP2$. Hence, $\phi$ must be one of the six permutations listed in Lemma~\ref{TightPathPresent}.  

Assume first that $\e{r}, \e{g} \in \Hef$. By Lemma~\ref{zFixed} we know that $\phi(\V{z}) = \V{z}$ and thus $\{\phi(\e{r}), \phi(\e{g})\} \subseteq \{\e{r}, \e{g}\}$. Therefore $\phi(\V{6}) = \V6$ holds by Observation~\ref{obs::H5properties}(1). This implies that $\phi$ is the identity since this is the only permutation listed in Lemma~\ref{TightPathPresent} which maps $\V6$ to itself.

Assume then that $\e{a} \in \Hef$. This implies that $\phi$ is the identity since this is the only permutation listed in Lemma~\ref{TightPathPresent} which maps $\e{a}$ to an edge of $\mathcal{H}$.
\end{proof}

\begin{lemma} \label{Property4proof}
Let $\phi : V(\mathcal{H}_{ef}) \rightarrow V(\mathcal{H})$ be a monomorphism. Then $\phi$ is the identity.
\end{lemma}  

\begin{proof}
Let $e'$ and $f'$ denote the two edges of $\mathcal H\setminus \phi (\Hef)$. Suppose for a contradiction that $\phi$ is not the identity. Observe that this implies that $\phi^{-1}$ is a monomorphism from $\phi(\Hef)$ to $\mathcal{H}$ which is not the identity. 

Since $\phi$ is not the identity, it follows from Lemma~\ref{94and59} that $\{\e{4,9}, \e{5,9}\} \cap \{e,f\} \neq \emptyset$. By  Lemma~\ref{94or59} we then infer that $\{\e{r}, \e{g}, \e{a}\} \cap \{e,f\} = \emptyset$. Similarly, since $\phi^{-1}$ is a monomorphism which is not the identity, it follows from Lemma~\ref{94and59} that $\{\e{4,9}, \e{5,9}\} \cap \{e', f'\} \neq \emptyset$ and from Lemma~\ref{94or59} that $\{\e{r}, \e{g}, \e{a}\} \cap \{e', f'\} = \emptyset$.  

By Lemma~\ref{zFixed} we know that $\phi(\V{z}) = \V{z}$ and thus $\{\phi(\e{r}), \phi(\e{g})\} \subseteq \{\e{r}, \e{g}\}$. Therefore $\phi(\V{6}) = \V6$ holds by Observation~\ref{obs::H5properties}(1). By Lemma~\ref{rg} we know that $\phi(\e{r}) = \e{g}$ and $\phi(\e{g}) = \e{r}$, which implies that $\phi(x) \neq x$ for every $x \in V(\mathcal{H}) \setminus \{\V{z}, \V{6}\}$.

Since $\e{g}, \e{a}$ and $\e{4,9}$ are the only edges which do not contain $\V5$, and $\{e,f\} \setminus \{\e{g}, \e{a}, \e{4,9}\} \neq \emptyset$, it follows that $d_{\Hef}(\V5) \leq 5$. Since $\{e', f'\} \setminus \{\e{g}, \e{a}, \e{4,9}\} \neq \emptyset$, an analogous  argument shows that $d_{\phi(\Hef)}(\V5) \leq 5$.

Suppose for a contradiction that $\e{5,9} \in \{e,f\}$. Then $d_{\Hef}(\V4) \geq 6$ and thus $d_{\phi(\Hef)}(\phi(\V4)) \geq 6$ as well. Since, as noted above, $d_{\phi(\Hef)}(\V5) \leq 5$, it follows from Table~\ref{DegreeTable} that $\phi(\V4) = \V{4}$. However, this contradicts  the fact that $\phi$ does not fix any vertex of $V(\mathcal{H}) \setminus \{\V{z}, \V{6}\}$. It follows that $\e{4,9} \in \{e,f\}$. An analogous argument shows that $\e{4,9} \in \{e',f'\}$ as well.

Suppose for a contradiction that $\e{1,5} \notin \{e,f\}$. Then $|\phi(\e{1,5}) \cap \e{g}| = |\phi(\e{1,5}) \cap \phi(\e{r})| = |\e{1,5} \cap \e{r}| = 3$ holds by Observation~\ref{obs::monomorphisms}(d). Since $\e{4,9}$ is the only edge of $\mathcal{H}$ which intersects $\e{g}$ in $3$ vertices, it then follows that $\phi(\e{1,5}) = \e{4,9}$. However, this contradicts the fact that $\e{4,9} \in \{e',f'\}$. An analogous argument shows that $\e{1,5} \in \{e',f'\}$ as well.

We have thus shown that $\{e,f\} = \{e',f'\} = \{\e{1,5}, \e{4,9}\}$. Hence, $P = (\e{2,6}, \e{3,7}, \e{4,8}, \e{5,9})$ is the unique tight path of length $4$ in $\Hef$ and in $\phi(\Hef)$. Since $\phi(\V{z}) = \V{z}$ and since $\phi(P) = P$ holds by Observation~\ref{obs::monomorphisms}(b) it follows that $\phi(\V1) = \V1$ contrary to $\phi$ not fixing any vertex of $V(\mathcal{H}) \setminus \{\V{z}, \V{6}\}$.
\end{proof}

\section{Concluding remarks and open problems} \label{sec::openprob}

As noted in the introduction, this paper originated from Beck's open problem of deciding whether $\mathcal{R}(K_q, \aleph_0)$ is a draw or FP's win. While it would be very interesting to solve this challenging problem, there are several natural intermediate steps one could make in order to improve one's understanding of the problem. In this paper we constructed a $5$-uniform hypergraph ${\mathcal H}_5$ such that $\mathcal{R}^{(5)}(\mathcal{H}_5, \aleph_0)$ is a draw, thus refuting the intuition that, due to strategy stealing and Ramsey-type arguments, $\mathcal{R}^{(k)}(\mathcal{H}, \aleph_0)$ is FP's win for every $k$ and every $k$-graph $\mathcal{H}$. It would be interesting to replace ${\mathcal H}_5$ with a graph.

\begin{question} \label{q::graph}
Is there a graph $G$ such that $\mathcal{R}^{(2)}(G, \aleph_0)$ is a draw?
\end{question} 

Our proof that $\mathcal{R}^{(5)}(\mathcal{H}_5, \aleph_0)$ is a draw, relies heavily on the fact that $\mathcal{H}_5$ has a vertex of degree $2$. Since this is clearly not the case with $K_q$, for $q \geq 4$, it would be interesting to determine whether this condition is necessary.

\begin{question} \label{q::minimumDegree}
Given an integer $d \geq 3$, is there a $k$-graph $\mathcal{H}$ such that $\delta(\mathcal{H}) \geq d$ and $\mathcal{R}^{(k)}(\mathcal{H}, \aleph_0)$ is a draw?  
\end{question}   

Another important ingredient in our proof that $\mathcal{R}^{(5)}(\mathcal{H}_5, \aleph_0)$ is a draw, is the fact that SP can build ${\mathcal H}_5 \setminus \{z\}$ very quickly. A similar idea was used in~\cite{FH} and in~\cite{FHkcon} to devise explicit winning strategies for FP in various natural strong games. On the other hand, it was proved by Beck in~\cite{BeckFast} that building a copy of $K_q$ takes time which is at least exponential in $q$. Intuitively, not being able to build a winning set quickly, should not be beneficial to FP. This leads us to raise the following question.

\begin{question} \label{q::slow}
Is there a $k$-graph $\mathcal{H}$ with minimum degree at least $3$ such that $\mathcal{R}^{(k)}(\mathcal{H}, \aleph_0)$ is a draw and, for every positive integer $n$, FP cannot win $\mathcal{R}^{(k)}(\mathcal{H}, n)$ in less than, say, $1000 |V(\mathcal{H})|$ moves?
\end{question}

\section*{Acknowledgment}

Part of the research presented in this paper was conducted during the two joint Free University of Berlin--Tel Aviv University workshops on Positional Games and Extremal Combinatorics. The authors would like to thank Michael Krivelevich and Tibor Szab\'o for organizing these events.

\end{document}